\documentclass[11pt]{amsart}
\usepackage{amsmath,amssymb, graphicx, amscd,latexsym,comment}
\makeatletter
\newtheorem{Theorem}{Theorem}
\newtheorem{Lemma}[Theorem]{Lemma}
\newtheorem{Corollary}[Theorem]{Corollary}
\newtheorem{Proposition}[Theorem]{Proposition}

\newtheorem{Remark}[Theorem]{Remark}

\newtheorem{Example}[Theorem]{Example}

\newtheorem{Problem}[Theorem]{Problem}
\newtheorem{Observation}[Theorem]{Observation}
\newtheorem{Assertion}[Theorem]{Assertion}

\newcommand{\eps}{\varepsilon}

\newcommand\vphi{\varphi}

\newcommand\al{\alpha}

\newcommand\si{\sigma}
\newcommand\be{\beta}

\newcommand\ga{\gamma}

\newcommand\de{\delta}

\newcommand\BC{ {\mathbb C}}

\newcommand\bfw{\mbox {\bf  w}}

\newcommand\bfz{\mbox {\bf  z}}

\newcommand\inv{^{-1}}

\def\mapright#1{\smash{\mathop{\longrightarrow}\limits^{{#1}}}}

\def\mapdown#1{\Big\downarrow\rlap{$\vcenter{\hbox{$#1$}}$}}

\def\inv{^{-1}}

\begin{document}
\title[ On the  roots of an extended Lens equation and an application
]
{On the  roots of an extended Lens equation and an application}

\author
[M. Oka ]
{Mutsuo Oka }
\address{\vtop{
\hbox{Department of Mathematics}
\hbox{Tokyo  University of Science}
\hbox{1-3 Kagurazaka, Shinjuku-ku}
\hbox{Tokyo 162-8601}
\hbox{\rm{E-mail}: {\rm oka@rs.kagu.tus.ac.jp}}
}}

\begin{abstract}
We consider zero points of   a generalized   Lens  equation 
$L(z,\bar z)=\bar z^m-{p(z)}/{q(z)} $ and also harmonically splitting Lens type equation
$L^{hs}(z,\bar z)=r(\bar z)-p(z)/q(z)$ with $\deg\, q(z)=n,\,\deg\,p(z)<n$  whose numerator is a  mixed polynomials, say $f(z,\bar z)$,
of degree $(n+m; n,m)$.
To such a polynomial, we associate a strongly mixed weighted homogeneous polynomial $F(\bfz,\bar \bfz)$
of  two variables and we show  the topology of Milnor fibration of $F$ is described by the number of roots of $f(z,\bar z)=0$.

\end{abstract}
\maketitle

\maketitle

\section{Introduction}
Consider  a mixed polynomial of one  variable $f(z,\bar z)=\sum_{\nu,\mu}a_{\nu,\mu}z^\nu{\bar z}^\mu$.
 We denote the set of roots of $f$ by
$V(f)$. Assume that $z=\alpha$ is an isolated zero of $f=0$. 
Put $f(z,\bar z)=g(x,y)+ih(x,y)$ with $z=x+iy$. 
A root $\alpha$ is called {\em  simple} if the Jacobian $J(g,h)$  is 
not vanishing at $z=\al$. We call $\al$ an orientation preserving or positive  (respectively   orientation reversing, or negative), if the Jacobian $J(g,h)$ is positive (resp. negative) at $z=\al$.

There are two basic questions.
\begin{enumerate}
\item 
Determine the number of roots with sign.
\item Determine the number of roots without  sign.
\end{enumerate}
\subsection{Number of roots with sign}
Let $C$ be a mixed projective curve of polar degree $d$ defined by a strongly  mixed homogeneous 
polynomial  $F(\bfz,\bar\bfz),\,\bfz=(z_1,z_2,z_3)$ of radial degree $d_r=d+2s$  and let $L=\{ z_3=0\}$ be a line
in $\mathbb P^2$.  We assume that $L$ intersects $C$ transversely.
\begin{Proposition} (Theorem 4.1, \cite{MC})
Then the fundamental class $[C]$ is mapped to $d[\mathbb P^1]$ and 
thus the intersection  number $[C]\cdot[L]$ is given by $d$. This is also 
given by  the number of the roots of $F(z_1,z_2,0)=0$ in $\mathbb P^1$ counted with sign.
\end{Proposition}
We assume that the point at infinity $z_2=0$ is not in the intersection $C\cap L$ and use the affine coordinate
$z=z_1/z_2$. Then $C\cap L$ is described by the roots of the mixed polynomial
$f(z):=F(z,1,0)$ which is written as
\[
f(z)=z^{d+s}\bar z^s+\text{(lower terms)}=0
\]
with respect to the mixed degree.
The second term is a linear combination
of monomials $z^a\bar z^b$ with $a+b<d+2s, a\le d,\,b\le s$.

Generic  mixed polynomials do not come from 
 mixed projective curves through  a holomorphic line section as above. 
The following is useful to compute the number of  zeros with sign of such polynomials.
Let $f(z,\bar z)$ be a given mixed polynomial of one variable.
 we consider the filtration by the degree:
\[\begin{split}
 f(z,\bar z)&=f_{ d}(z,\bar z)+f_{ d-1}(z,\bar z)+\cdots+
f_{0}(z,\bar z)
\end{split}
\]
with $d( k)< d$. 
Here
  $
  f_\ell(z,\bar z):=\sum_{\nu+\mu=\ell}c_{\nu,\mu}z^\nu\bar z^\mu.
 $
Note that we have  a unique 
 factorization of $f_{d}$  as follows.
\begin{eqnarray*}
 f_{ d}(z,\bar z)&=c z^{ p}\bar z^{ q}\prod_{j=1}^{s}(z+\ga_j \bar
 z)^{\nu_j},\quad\\
&p+q+\sum_{j=1}^{s}\nu_j=d,\,c\in \BC^*.
\end{eqnarray*}
where $\ga_1,\dots, \ga_{s}$
are mutually distinct non-zero complex numbers.
We say that $f(z,\bar z)$ is {\em admissible at infinity}
 if
$|\ga_j|\ne 1$ for $j=1,\dots, s$ .
For non-zero complex number $\xi$, we put 
\[\eps(\xi)=\begin{cases}
1\quad &|\xi|<1\\
-1\quad &|\xi|>1\end{cases}
\]
and we consider the following  integer:
\[
\be(f):= p- q+\sum_{j=1}^{s} \eps(\ga_j)\nu_j,
\]
The following equality holds.
\begin{Theorem} (\cite{MixIntersection})\label{main}
Assume that $f(z,\bar z)$ is an admissible mixed polynomial
 at infinity.
 Then the total number of roots with sign is equal to 
$\be(f)$.
\end{Theorem}

\begin{Remark}
Here if $\al$ is a non-simple root, we count the number with multiplicity.
The multiplicity is dfined by the local rotation number at $\al$ of the normalized Gauss mapping $S_\eps(\al)\to S^1$, $z\mapsto f(z)/|f(z)|$.
\end{Remark}
\subsection{Number of roots ignoring the sign}
In this paper, we are interested in the total number  of $V(f)$ which we denote by  $\rho(f)$, the cardinality of
$\sharp V(f)$ for  particular classes  of mixed polynomials ignoring the sign.
The notion of the multiplicity is not well defined for a root without sign. Thus we assume that roots are all simple.
The problem is that $\rho(f)$ is not described by the highest degree part $f_d$, which was the case for
the number of roots with sign  $\be(f)$.
We will an example of mixed polynomial below $\rho(f)=n^2$.
Another example is known by Wilmshurst  (\cite{Wilmshurst}).
\begin{Example}
Let us consider the Chebycheff polynomial $T_n(x)$.
The graph has two critical values 1 and -1 and the roots of $T_n(x)=0$ is in the interval
$(-1,1)$.
Consider  a polynomial
\[
F(x,y)=(y-T_n(x)+i(x-aT_n(by)),\quad a,b \gg 1.
\]
By the assumption,  $F=0$
has $n^2$ roos in $(-1,1)\times (-1,1)$.
Consider $F$ as a mixed polynomial by substituting $x,y$ by
$x=(z+\bar z)/2,\,y=-i(z-\bar z)/2$. 
This example gives an extreme case for which the possible complex roots (by Bezout theorem)
of $\Re F=\Im F=0$  are all  real roots.
\end{Example}

\begin{figure}[htb,H]
\setlength{\unitlength}{1bp}
\includegraphics[width=7cm,height=6cm]{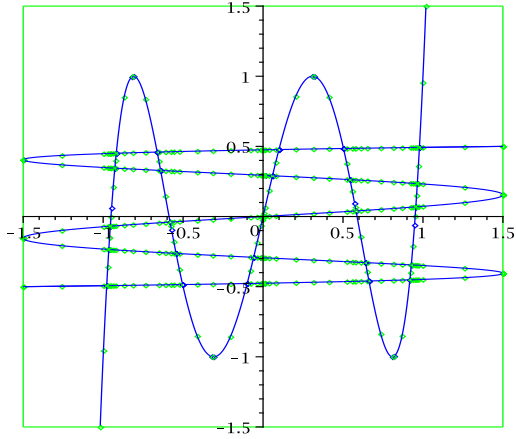}
\vspace{1cm}
\caption{Roots of $F(x,y)=0, n=5$}
\end{figure}

The above example shows implicitly that the behavior of the number of roots without sign  behaves very violently if we do not assume any assumption on $f$.

Consider  a mixed polynomial of one  variable $f(z,\bar z)=\sum_{\nu,\mu}a_{\nu,\mu}z^\nu{\bar z}^\mu$.  Put
\[\begin {split}
&\deg_z\,f:=\max\{\nu\,|\, a_{\nu,\mu}\ne 0\}\\
&\deg_{\bar z}\,f:=\max\{\mu\,|\,a_{\nu,\mu}\ne 0\}\\
&\deg\,f:=\max\{\mu+\nu\,|\,  a_{\nu,\mu}\ne 0\}
\end{split}
\]
We call $\deg_z\,f,\,   \deg_{\bar z}\,f,\, \deg\,f$  {\em the holomorphic degree} , {\em  the anti-holomorphic degree}
and {\em the mixed degree} of $f$
 respectively. 
We  consider the following subclasses  of mixed polynomials:
\begin{eqnarray*}
&M(n+m;n,m):=&\{f(z,\bar z)\,|\, \deg\,f=n+m,\,\deg_z\,f=n,\,\deg_{\bar z}\,f=m\},\\
&L(n+m;n,m):=&\{\bar z^m q(z)-p(z)\,|\, \deg_z q(z)=n,\,\,\deg_z p(z)\le n\},\\
&L^{hs}(n+m;n,m):=&
\{r(\bar z) q(z)-p(z)\,| \deg_{\bar z} r(\bar z)=m,\,\\
&&\qquad\qquad\qquad \deg_z, q(z)=n,\,\,\deg_z p(z)\le n\}.
\end{eqnarray*}
where $p(z),q(z)\in \mathbb C[z],r(\bar z)\in \mathbb C[\bar z]$. We have  canonical inclusions:
\[
L(n+m;n,m)\,\,\subset L^{hs}(n+m;n,m)\,\,\subset M(n+m;n,m).
\]
The class $L(n+m;n,m), L^{hs}(n+m;n,m)$ come from harmonic functions 
\[
\bar z^m-\frac{p(z)}{q(z)},\,\, r(\bar z)-\frac{p(z)}{q(z)}
\]
as their  numerators. Especially $L(n+1;n,1)$ corresponds to the lens equation.
We call $\bar z^m-\frac{p(z)}{q(z)}=0$ 
and  $ r(\bar z)-\frac{p(z)}{q(z)}=0$ 
{\em a generalized lens equation}
and  {\em a harmonically splitting lens type equation} respectively.
The corresponding numerators are called a generalized lens polynomial and
 a harmonically splitting lens type polynomial respectively.
The  polynomials which attracted us in this paper are these classes. 
We thank to A. Galligo 
for sending us  their paper  where we learned  this problem ({\cite{Elkadi-Galligo}).
\subsection{Lens equation 
}
The following equation is known as the lens equation.
\begin{eqnarray}\label{lens-rational1}
L(z,\bar z)=\bar z-\sum_{i=1}^n  \dfrac{\si_i}{z-\al_i}=0,\quad \si_i,\al_i\in \mathbb C^*.
\end{eqnarray}
We identify the left side rational function  with the mixed polynomial given by its numerator
\[
L(z,\bar z)\prod_{i=1}^n(z-\al_i)\in M(n+1;n,1).
\]
throughout this paper. The real and imaginary part of this polynomial
are polynomials of $x,y$ of degree $n+1$.
Unlike the previous example, $ \rho(f)$ is much more smaller than $(n+1)^2$.
This type of  equation  is studied for astorophisists. 
For more explanation from astrophiscal viewpoint, see for example
Petters-Werner \cite{Petters-Werner}. The lens equation can be witten as 
\begin{eqnarray}\label{lens-rational2}
L(z,\bar z)&:=&\bar z- \vphi(z),\quad \vphi(z)=\dfrac{p(z)}{q(z)}=0,\\
&&\quad \deg\, p(z)\le n,\,\deg\, q(z)=n.\notag
\end{eqnarray}
A slightly  simpler equation 
 is 
\begin{eqnarray}\label{lens-polynomial}
L'(z,\bar z):=\bar z-p(z),\quad \deg_z\,p=n.
\end{eqnarray}
Both equations are studied using complex dinamics.
Consider the function $r:\mathbb P^1\to\mathbb P^1$ defined 
$r(z)=\overline{  \vphi(\overline{\vphi(z)} )}$. It is easy to see that $r$ is a rational mapping
of degree $n^2$.
Observe that $z$ is a root of $L(z)=0$ , then $z$ is a fixed point of $r(z)$, that is $z=r(z)$.
%
It is known that 
\begin{Proposition} \label{bound}
The number of zeros $\rho(L')$  of $L'$,  is bounded by  $3n-2$  
 Khavinson-\'Swiat\c{e}k \cite{Khavinson-Swiatek} and the  number of zeros
 $\rho(L)$ of $L$ is bounded by $5n-5$ by Khavinson-Neumann  \cite{Khavinson-Neumann}. 
\end{Proposition} 
  Bleher-Homma-Ji-Roeder has  determined the exact  range of $\rho(L)$:
\begin{Theorem}\label{Lens-range}(Theorem 1.2,\cite{Bleher})Suppose that the lens equation has only simple solutions. Then the set of   possible numbers of 
solutions is equal to
\[
\{n-1+2k\,|\, 0\le k\le 2n-2\}=
\{n-1,n+1,\cdots, 5n-7,5n-5\}.
\]
\end{Theorem}

 The estimation in Proposition \ref{bound} are optimal.
Rhie gave an explicit example of $f$ which satisfies $ \rho(f)=5n-5$ ( See Rhie \cite{Rhie},
 Bleher-Homma-Ji-Roeder  \cite{Bleher}, and also Theorem \ref{main1} below). Thus the inequality $\rho(f)\le 5(n-1)$ is optimal.
The minimum of $\rho$ is  $n-1$ and it can be obtained for example by $\bar z z^n-1$.

In the proof of Proposition \ref{bound}, the following principle in complex dinamics plays a key role.
\begin{Lemma}
Let $r$ be an rational function on $\mathbb P^1$. If $z_0$ is an  attracting or rationbally neutral fixed point, then $z_0$ attracts some critical point of $r$.
\end{Lemma}
Elkadi and Galligo studied this problem from computational point of view  to construct such a mixed polynomial explicitly and proposed the similar problem for generalized lens polynomials $L(n+m;n,m)$
 (\cite{Elkadi-Galligo}).

 \section{Relation of strongly polar weighted homogeneous polynomials and number of zeros without sign}
 Consider a strongly mixed weighted homogeneous polynomial $F(\bfz,\bar \bfz)$ of
two variables $\bfz=(z_1.z_2)$ with
 polar weight $P={}^t(p,q)$, $\gcd(p.q)=1$ and let $d_p,\,d_r$ be the polar and  radial degrees
respectively.
 Let 
 \[
 \mathbb C^*\times \mathbb C^2\to \mathbb C^2,\quad 
 (\rho,(z_1,z_2))\mapsto \rho\underset P\cdot (z_1,z_2):=(z_1\rho^p,z_2\rho^q)
 \] be the associated $\mathbb C^*$-action.
Recall that $F$ satisfies the Euler equality:
\[
F(r\exp(\theta i)\underset P\cdot( \bfz,\bar\bfz))=r^{d_r}\exp(d_p\theta i)F(\bfz,\bar\bfz).
\]
 A strongly mixed polynomial is the case where the weight is 
 the canonical weight ${\bf 1}:={}^t(1,1)$.
Consider the global Milnor fibration $F:\,\mathbb C^2\setminus F\inv({\bf 0})\to \mathbb C^*$ and 
 let $M=\{\bfz\in \mathbb C^2|F(\bfz,\bar\bfz)=1\}$ be the Milnor fiber. 

 We assume further that $F$ is convenient.
  By the convenience assumption and the strong mixed weighted homogenuity,
 we can find some  integers $n,r$ such that 
 \[
 d_p=npq,\quad d_r=(n+2r)pq
 \]
 and  we can write 
$F(z_1,\bar z_1,z_2,\bar z_2)$ as a linear combination of monomials 
$z_1^{\nu_1}z_2^{\nu_2}\bar z_1^{\mu_1}\bar z_2^{\mu_2}$
 where the summation satisfies the equality
 \[(\nu_1+\nu_2)p+(\mu_1+\mu_2)q=d_r\quad (\nu_1-\nu_2)p+(\mu_1-\mu_2)q=d_p.
 \]
 In particular, we see that 
 the coefficients of $z_1^{(n+r)q}\bar z_1^{rq}$ and $z_2^{(n+r)p}\bar z_2^{rp}$ are non-zero
 and any other monomials satisfies
 \[\nu_1,\nu_2\equiv 0\,\mod\, q,\quad \mu_1,\mu_2\equiv 0\, \mod \, p.\]
The monodromy mapping $f:M\to M$ is defined by
\[
h:M\to M,\quad \bfz\mapsto  \exp(2\pi i/npq)\underset P\cdot \bfz=\exp(2\pi i/nq) z_1,\exp(2\pi i/np)z_2).
\]
 Thus there exists a strongly mixed homogeneous polynomial $G(\bfw,\bar\bfw),\,\bfw=(w_1,w_2)$
 of polar degree $n$ and radial degree $(n+2r)$
 such that
 \[
 F(\bfz,\bar\bfz)=G(z_1^q,{\bar z}_1^q,z_2^p,{\bar z}_2^p).
 \]
The curve $F=0$ is invariant under the above $\mathbb C^*$-action.
 Let $\mathbb P^1(P)$ be the weighted projective line
which is the quotient space of $\mathbb C^2\setminus\{\bf 0\}$ by the above action. 
 It has two singular points $A=[1,0]$ and $B=[0,1]$ (if $p,q\ge 2$) and  the complement
$U:=\mathbb P^1(P)\setminus\{A,B\}$ is isomorphic to 
 $\mathbb C^*$ with coordinate $z:=z_1^q/z_2^p$.
Note that $z$ is well defined on $z_2\ne 0$.
 The zero locus of $f$ in $\mathbb P^1(P)$, $V(f)$,   does not contain $A,B$ and it is defined
on $U$ by the mixed polynomial
 $f(z,\bar z)=0$ where $f$ is defined by the equality:
 \[\begin{split}
 f(z,\bar z)&:= F(\bfz,\bar\bfz)/(z_2^{(n+r)p}\bar z_2^{rp})\\
 &=c\,z^{n+r}\bar z^r+\sum_{i,j} a_{i,j}z^i\bar z^j
 \end{split}
  \]
  where the summation is taken for $i\le n+r,\,j\le r$ and $i+j<n+2r$
and $c\ne 0$ is the coefficient of $z_1^{n+r}\bar z_1^r$ in $F$.
Note also that $g(z)=f(z)$ where 
\[
g(w):=G(w_1,\bar w_1,w_2,\bar w_2)/(w_2^{n+r}\bar w_2^r),\quad w=w_1/w_2.
\]
Thus in these affine coordinates $z,w$,
we have
\[
\begin{split}
&z_1^{q n_1}{\bar{ z}_1}^{qn_2}z_2^{pm_1}{\bar z_2}^{pm_2}/(z_2^{(n+r)p}\bar z_2^{rp})=z^{n_1}{\bar z}^{n_2},\\
&w_1^{n_1}{\bar w_1}^{n_2}w_2^{m_1}{\bar w_2}^{m_2}/(w_2^{n+r}\bar w_2^r)=w^{n_1}{\bar w}^{n_2}
\end{split}
\]
This implies that $f(z)=g(z)$,  the number of points of $V(f(z))$ and $V(g(w))$ are equal in their respective projective spaces
and 
\[
f=g\in M(n+2r;n+r,r). 
\]
  The associated $\mathbb C^*$-action to $G(\bfw,\bar\bfw)$ is the canonical linear action and we simply denote
  it as $\rho\cdot \bfw$ instead of $\rho\underset{\bf 1}\cdot \bfw$.
 Let $M(G)$ be the Milnor fiber of $G$ and let $\mathbb P^{1}$ be the usual projective line. 
 The monodromy mapping $h_G:M(G)\to M(G)$ of $G$ is given by
 $h_G(\bfw)=\exp(2\pi i/n)\cdot \bfw$.
 Then we have a canonical diagram
 \[
 \begin{matrix} 
C^2&\mapright{\vphi_{q,p}}&\mathbb C^2\\
\uparrow&&\uparrow\\
 M&\mapright{\vphi_{q,p}}& M(G)\\
 \mapdown{\pi}&&\mapdown{\pi'}\\
 \mathbb P^1(P)\setminus V(f)&\mapright{\bar \vphi_{q,p}}&\mathbb P^1\setminus V(g)
 \end{matrix}
 \] $\pi$ is a $\mathbb Z/d_p\mathbb Z$-cyclic covering branched over $\{A,B\}$,( $d_p=npq$)
 while $\pi'$ is a $\mathbb Z/n\mathbb Z$-cyclic covering without any branch locus.
 $\vphi_{q,p}$ is defined $\vphi_{q,p}(z_1,z_2)=(z_1^q,z_2^p)$ which satisfies
 $\vphi_{q,p}(\rho\underset P\cdot\bfz)= \rho^{pq}\cdot (z_1^q,z_2^p),\,\rho\in S^1$
 and thus $\vphi_{q,p}\circ h=h_G\circ \vphi_{q,p}$ as we have
 \[\begin{split}
\vphi_{q,p}(h(\bfz))&=
\vphi( (\exp(2\pi i/nq)z_1,\exp(2\pi i/np)z_2 )\\
&=(\exp(2\pi i/n) z_1^q,\exp(2\pi i/n)z_2^p)=
h_G(\vphi_{q,p}(\bfz)).
\end{split}
\]
 The mapping $\bar\vphi_{q,p}$ is canonically induced by $\vphi_{q,p}$
 and we observe that
 $\bar\vphi_{q,p}$ gives  a bijection of
 \[\bar\vphi_{q,p}:\mathbb P^1(P)\setminus \{A,B\} \to \mathbb P^1\setminus\{\bar A,\bar B\}\]
and   it induces an bijection of $V(f)$ and $V(g)$.
Here $\bar A=[1:0]$ and $\bar B=[0,1]$.
 Recall that by \cite{OkaMix,VO}, we have
 \begin{Proposition}\label{Correspondence1}
 \begin{enumerate}
 \item$\chi(M(G))=n(2-\rho(g))$.
 \item $\chi(M)=-npq\rho(f)+n(p+q).$
\item The links $K_F:=F\inv(0)\cap S^3$ and $K_G:=G\inv(0)\cap S^3$ have the same number of components and it is given by $\rho(f)$.
\end{enumerate}
 \end{Proposition}
\begin{proof}
The assertion follows from a simple calculation of Euler characteristics. 
(1) is an immediate result that $M(G)\mapright{\pi'} \mathbb P^1\setminus V(G)$ is  an $n$-fold cyclic covering.
(2) follows from the following.
\[\pi: M\cap \mathbb C^{*2}\to \mathbb P^1(P)\setminus ( \{A,B\}\cup V(F)\] is an $npq$-cyclic covering   while
$M\cap\{z_1=0\}$ and $M\cap\{z_2=0\}$  are $np$ and $nq$ points respectively.
Thus 
\begin{eqnarray*}
\chi(M)&=&\chi(M\cap \mathbb C^{*2})+\chi(M\cap\{z_1=0\})+\chi(M\cap \{z_2=0\})\\
&=&npq(-\rho(f))+np+nq
\end{eqnarray*}
The link components of $K_F$ and $K_G$ are $S^1$ invariant and the assertion (3) follows from this observation.
\end{proof}
 The correspondence $F(\bfz,\bar\bfz)\mapsto f(z)$ is reversible. Namely we have
\begin{Proposition}\label{Correspondence2}
 For a given $f(z,\bar z)\in M(n+m;n,m)$ and any weight vector $P={}^t(p,q)$, we can define
 a strongly mixed weighted homogeneous polynomial  of two variables $\bfz=(z_1,z_2)$ with weight $P$ by
 \[
 F(\bfz,\bar\bfz):=f(z_1^q/z_2^p,\bar z_1^q/\bar z_2^p)z_2^{pn}{\bar z_2}^{pm}.
 \]
The polar degree and the radial degree of $F$ are  $(n-m)pq$ and $(n+m)pq$ respectively.
The coefficient of  $z_1^{nq}\bar z_1^{qm}$ in $ F$ is the same as that of $z^{n}\bar z^m$.

 If $f$ has non-zero constant term,
$F$ is convenient polynomial.
The correspondence
\[
F(\bfz,\bar \bfz)\mapsto f(z,\bar z),\quad f(z,\bar z)\mapsto F(\bfz,\bar \bfz)
\]
are inverse of the other.
\end{Proposition}
\begin{proof}
In fact, the monomial $z^i{\bar z}^j,\,i+j\le n+m,\,i\le n,\,j\le m$
changes into 
$z_1^{qi}{\bar z_1}^{jq }z_2^{p(n-i)}{\bar z_2}^{p(m-j)}$.
In particular,
\[
z^n{\bar z}^m\mapsto z_1^{qn}{\bar z_1}^{qm},\quad
1\mapsto z_2^{pn}{\bar z_2}^{pm}.
\]
\end{proof}
It is well-known that the Milnor fibration of a weighted homogeneous polynomial $h(\bfz)\in \mathbb C[z_1,\dots, z_n]$ with an islated singularity at the origin is described by the weight and the degree by Orlik-Milnor \cite{Orlik-Milnor}. This  assertion is not true for a mixed weighted homogeneous polynomials.

Let 
\[\tilde M(n+m;n,m;P),\,\, {\tilde L}^{hs}(n+m;n,m;P),\,\,\tilde L(n+m;n,m;P)\]
 be the space of strongly mixed weighted homogeneous
convenient  polynomials of two variables  with weight $P=(p,q),\,\gcd(p,q)=1$ and with isolated singularity at the origin
which corresponds to $M(n+m;n,m), L^{hs}(n+m;n,m),L(n+m;n,m)$
 respectively through Proposition \ref{Correspondence1}
and Proposition \ref{Correspondence2}.  For $P=(1,1)$, we simply write as 
\[\tilde M(n+m;n,m),\,\, {\tilde L}^{hs}(n+m;n,m),\,\,\tilde L(n+m;n,m)\]
\begin{Proposition}\label{correspondebce3}
The moduli spaces $\tilde M(n+m;n,m;P)$, $ {\tilde L}^{hs}(n+m;n,m;P)$,
$\tilde L(n+m;n,m;P)$
are isomorphic to the moduli spaces
$M(n+m;n,m)$,  $L^{hs}(n+m;n,m)$, $L(n+m;n,m)$ respectively.
\end{Proposition}
As the above  moduli spaces  do not depend on the weight $P$ (up to isomorphism), we only consider hereafter strongly mixed homogeneous polynomials.
 Assume that  two polynomials $F_1, F_2$ are in a same connected coponent. then their
Milnor fibrations, are equivalent. Thus 
\begin{Corollary} Assume that $F_1, F_2\in \tilde M(n+m;n,m) $ has different number of  link components $\rho(f_1),\, \rho(f_2)$. Then they belongs to different connected components of $\tilde M(n+m;n,m)$.
In particular, the number of the connected components of $ \tilde M(n+m;n,m)$ is not smaller than the number of 
\newline
 $\{\rho(f)\,|\,f\in M(n+m;n,m)\}$.
\end{Corollary}
 \begin{Remark} 
For a fixed number $\rho$ of link components, we do not know if the subspace of the moduli space with  link number $\rho$ is connected or not.
 \end{Remark}

 \begin{Example}\label{degree2}
Consider a strongly mixed homogeneous polynomial $F$ of polar degree 1 and radial degree 3.
Namely $f\in M(3;2,1)$.
 Its possible link components are 1,3,5. $1$ and $3$ are given in Example 59, \cite{OkaMix}.
 An example of 5 components are given by Bleher-Homma-Ji-Roeder (\cite{Bleher}). For example, we can take
 \[\begin{split}
 f(z)&=\bar z(z^2-1/2)-z+1/30\\
 F(\bfz,\bar\bfz)&=\bar z_1 (z_1^2-z_2^2/2)-(z_1z_2-z_2^2/30)\bar z_2
 \end{split}
 \]
 \end{Example}
\section{Extended lens equation}
\subsubsection{Extended lens equation}
One of the main purposes of this paper is to study the number of zeros  of the following extended lens equation for a given $m\ge 1$ and its perturbation.
\[
 L(z,\bar z)=\bar z^m- \dfrac{p(z)}{q(z)},\quad \deg\,q=n,\, \deg\,p
 \le n.
\]
The corresponding mixed polynomail is in $L(n+m;n,m)\subset M(n+m;n,m)$.
We will construct a mixed polynomial
for which 
the example of Rhie is extended.
However a  simple generalization of Proposition \ref{bound} seems not possible. The reason is the following.
Consider the function
\[\vphi:=\root m\of {{\frac{\overline{ p(z)}}{\overline{ q(z)}}}}
\]
and the composition
$\psi:=\vphi\circ \vphi$.   $\psi$ is a locally holomorphic function but the point is that 
$\vphi $ and $\psi$ are  multi-valued functions, not single valued if $m\ge 2$.
Thus we do not know any meaningful upper bound of $\rho(L)$.
\subsection{A symmetric case}
Here is one special case where we can say more.
Suppose  that $m$ divide $n$ and put $n_0=n/m$.
Assume that $p(z)/q(z)$ is $m$-symmetric, in the sense that there exists polynomials
$p_0(z),q_0(z)$ so that 
$p(z)=p_0(z^m)$ and $q(z)=q_0(z^m)$.  We assume that 
$p_0(0)\ne 0$.  In this case, we can consider the lens equation
\begin{eqnarray}
L_0(z,\bar z)&:=&\bar z- \vphi_0(z),\quad \vphi_0(z)=\dfrac{p_0(z)}{q_0(z)},\\
L(z,\bar z)&:=&\bar z^m- \vphi(z),\quad \vphi(z)=\dfrac{p(z)}{q(z)}.
\end{eqnarray}
As $L(z,\bar z)=L_0(z^m,{\bar z}^m)$,
there is  $m:1$ correspondence  between the non-zero roots of $L$ and $L_0$.
Thus by Proposition \ref{bound},
we have
\[\rho(L)=m\rho (L_0)\le  m(5n_0-5)=5n-5m.\]
\begin{Corollary}\label{case-n=2m}Suppose that $n=2m$ and let 
$f(z,\bar z)=\bar z-\frac{z-1/30}{z^2-1/2}$ as in Example \ref{degree2}.
Put $f_{2m}(z)=f(z^m,\bar z^m)$. Then $\rho(f_{2m})=5m$ and
the corresponding strongly mixed  homogeneous polynomial $F_{2m}$ is contained  in $\tilde L(3m;2m,m)$.
\end{Corollary}
\subsection{Generalization of the Rhie's example}
So we will try to generalize the example of Rhie for the case $m\ge 2$ without assuming  $n\equiv 0\,\mod\, m$.
First we consider the following extended Lens equation:
\begin{eqnarray}\label{lens(n,m)}
 \ell_{n,m}(z,\bar z)=\bar z^m- \dfrac{z^{n-m}}{z^n-a^n}=0,\quad n>m>0,\,\, a\in
 \mathbb R_+
\end{eqnarray}
Hereafter by abuse of notation, we also denote the corresponding mixed polynomial
(i.e., the numerator)
by  the same $\ell_{n,m}(z,\bar z)$.
For the study of $V(\ell_{n,m})\setminus \{0\}$, we may consider equivalently the following: 
\begin{eqnarray}
|z|^{2m}- \dfrac{z^n}{z^n-a^n}=0.
\end{eqnarray}
This can be rewritten as
\[
z^n(|z|^{2m}-1)=|z|^{2m}a^n.
\]
Thus we have 
\begin{Proposition}
\label{positivityofroot} Take a  non-zero root $z$ of $\ell_{n,m}=0$.
If $a>0$ and $z\ne 0$,   then $|z|\ne 1$ and $z^n$ is a real number.
Thus $z^{2n}$ is a positive real number.
\end{Proposition}

Let us consider the half lines
 \[\begin{split}
\mathbb R_+(\theta):=&\{re^{i\theta}\,|\, r\ge 0\}
\end{split}
\]
and lines 
$L_\theta$ which  are the union of two half lines:
\[L_\theta
=\mathbb R_+(\theta)\cup \mathbb R_+(\theta+\pi).
\]
Put 
\[\begin{split}
L(n):=&\bigcup_{j=0}^{n-1} L_{{2\pi j}/n},\quad
L(n)':=\bigcup_{j=0}^{n-1} L_{(2j+1)\pi/n}\\
\mathcal L(2n):=&\bigcup_{j=0}^{2n-1} \mathbb R_+(j\pi/n)
=\{z\in \mathbb C\,|\, z^{2n}\ge 0\}.
\end{split}
\]
\begin{Observation}
\begin{enumerate}
\item If $n$ is odd,  $L_{2\pi j/n}=L_{(2j+n)\pi/n}$ and thus 
$L(n)=L(n)'$ and they consists of $n$ lines and $L(n)=\mathcal L(2n)$.
\item If $n$ is even, $L(n)\cap L(n)'=\{0\}$,
$\mathcal L(2n)=L(n)\cup L'(n)$ and lines of  $L(n)$ and $L(n)'$   are  doubled.
That is, each half line $\mathbb R_+(2\pi j/n)$  and  $\mathbb R_+(\pi (2j+1)/n)$
 appear twice in $L(n)$ 
 and 
in $L(n)'$ respectively.
\end{enumerate}
\end{Observation}
We identify  $\mathbb Z/n\mathbb Z$ with complex numbers which are $n$-th root of unity and we 
consider  the canonical  action of $\mathbb Z/n\mathbb Z\subset \mathbb C^*$ on $\mathbb C$
by  multiplication.
Thus it is easy to observe that 
\begin{Lemma}
$V(\ell_{n,m})$
is a subset of $\mathcal L(2n)$ and $V(\ell_{n,m})\cap L(n)$  and $V(\ell_{n,m})\cap L(n)'$ are  stable by the action of $\mathbb Z/n\mathbb Z$.
\end{Lemma}
For non-zero real  number solutions of (\ref{lens(n,m)}) are given by the roots of the following equation:
\begin{eqnarray}
\ell_{n,m}(z,\bar z)&=&z^{2m}- \dfrac{z^{n}}{z^n-a^n}=0,\quad z\in \mathbb R^*.
\end{eqnarray}
Equivalently
\[
\begin{cases}
z^n-a^n-z^{n-2m}=0,&\quad n> 2m\\
z^{2m-n}(z^n-a^n)-1=0,&\quad n\le 2m.
\end{cases}
\]
Note that for $n$ odd, $V(\ell_{n,m})\subset L(n)$ and  the generator $e^{2\pi i/n}$ of $\mathbb Z/n\mathbb Z$ acts cyclicly as
\[\begin{split}
& 
L_{2\pi j/n}\cap V(\ell_{n,m})\mapsto L_{2\pi (j+1)/n}\cap V(\ell_{n,m})\\
& 
L_{\pi (2j+1)/n}\cap V(\ell_{n,m})\mapsto L_{\pi (2j+3)/n}\cap V(\ell_{n,m})
\end{split}
\]
 For $n$ even, $V(\ell_{n,m})\subset L(n)\cup L(n)'$.
To consider the  roots on  $L(n)'$, we put 
$z=\exp(\pi(2j+1)i/n)\cdot u$ with $u\in \mathbb R^*$. Then $u$ satisfies
\begin{eqnarray}
u^{2m}-
 \dfrac{-u^n}{-u^n-a^n}=0,\,\,\text{if}\,\, u\in \mathbb R.\label{even n}
\end{eqnarray}
This  is equivalent to
\begin{eqnarray}
\begin{cases}
u^n+a^n-u^{n-2m}=0,&\quad n >2m\\
u^{2m-n}(u^n+a^n)-1=0,&\quad n\le 2m.\label{even n2}
\end{cases}
\end{eqnarray}

\subsection{Preliminary result before a bifurcation}
The first preliminary results is the following (Lemma \ref{ex-lens}, Lemma \ref{ex-lens2}).
\begin{Lemma}\label{ex-lens}
 If $n>2m$,  for a sufficiently small $a>0$,  $\rho(\ell_{n,m})= 3n$.
\end{Lemma}
\begin{proof} The proof is parallel  to  that of Rhie (\cite{Rhie, Bleher}.
We know that roots are on $L(n)$ or $L(n)'$ by Proposition 
\ref{positivityofroot}.
Consider  non-zero real  roots of $\ell_{n,m}(z,\bar z)=0$.
It satisfies the equality:
\begin{eqnarray}\label{n>2m-odd}
  z^m- \dfrac{z^{n-m}}{z^n-a^n}=0\iff  z^n-z^{n-2m}-a^n=0,\quad z\in
   \mathbb R\setminus \{0\}.
\end{eqnarray}
 (1) Assume that $n$ is odd.
Then the function $w=z^n-z^{n-2m}$ has three real points on the real
 axis,
$(-1,0),\,(0,0),\, (1,0)$ and the graph looks like Figure \ref{F2}.  As we see in  the Figure, they have one relative  maximum $\al>0$ and one relative minimum $-\al$.
Thus the horizontal line $w=a^n$ intersects with this graph at three points if $a^n<\al$.
\begin{figure}[htb] 
\setlength{\unitlength}{1bp} 
\centerline{\includegraphics[width=6cm]{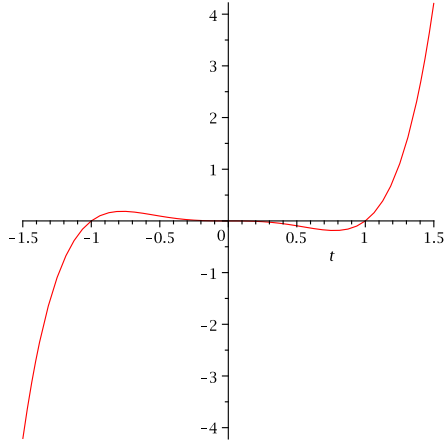}} 
\vspace{1cm}
\caption{Graph of $y=t^n-t^{n-2m}$,\,$n=5,m=1$} 
\label{F2}
\end{figure} 

Thus (\ref{n>2m-odd}) has three real roots for a
 sufficiently small $a$.
Now we consider the action of $\mathbb Z/n\mathbb Z$ on $V(\ell_{n,m})$,  we have $3n$ solutions 
on $V(\ell_{n,m})\cap L(n)$.  

\noindent
(2) Assume that $n$ is even.
In this case, we have to notice that the action of $\mathbb Z/n\mathbb Z$ on $V(f)\cap L(n)$
is  $2:1$ off the origin.

In this case,  the graph of $y=t^n-t^{n-2m}$
looks like Figure 3.
\begin{figure}[htb] 
\setlength{\unitlength}{1bp} 
\centerline{\includegraphics[width=6cm,angle=-90]{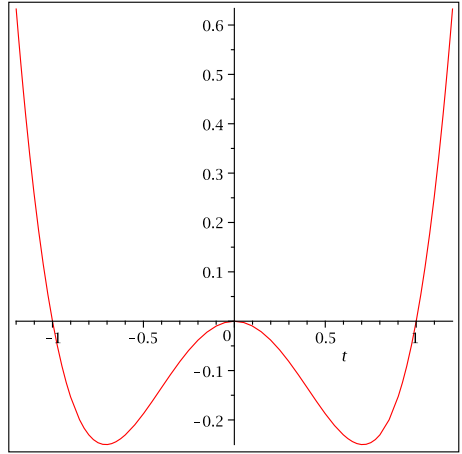}} 
\caption{Graph of $y=t^n-t^{n-2m}$,\,$n=4,m=1$} 
\end{figure} 
Thus for a sufficiently small $a>0$, $t^n-t^{n-2m}-a^n=0$ has two real roots.
Thus by the above remark,  it gives 
$2n/2=n$ roots on  $V(\ell_{n,m})\cap L(n)$.
Now we consider the roots on the line $L_{(2j+1)\pi/n}$,
$j=0,\dots, n-1$. Putting $z=u\zeta, \zeta^n=-1$ with $u$ being real, from (\ref{even n2}), we get 
the equality:
\begin{eqnarray}
 -u^n+u^{n-2m}-a^n=0.
\end{eqnarray}
The graph of $y=-t^n+t^{n-2m}$ is the mirror image of Figure 3 with
 respect to $t$-axis.
Thus $-t^n+t^{n-2m}-a^n=0$ has 4 real roots.
Counting all the roots on the lines in 
$L(n)'$, it gives 4n/2=2n roots.
Thus altogether, we get $3n$ roots.
\end{proof}

Now we consider the case $2m\ge n$.
\begin{Lemma}\label{ex-lens2}
If $2m\ge n>m$, $\rho(\ell_{n,m})=2n$ for a sufficiently small $a>0$.
\end{Lemma}
\begin{proof}

\noindent
(1a)  Assume that $2m>n$ and $n$ is odd.
The equation of the real solutions of (\ref{lens(n,m)} ) reduces to 
\begin{eqnarray}\label{2m>n}
z^{2m-n}(z^n-a^n)=1.
\end{eqnarray}
It is easy to see that there are two real solutions (one positive and one negative). See Figure 4.
\begin{figure}[htb] 
\setlength{\unitlength}{1bp} 
\centerline{\includegraphics[width=4cm,height=6cm, angle=-90]{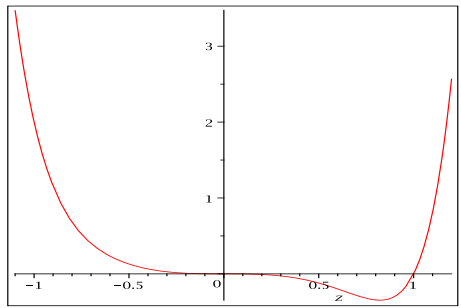}} 
\vspace{.6cm}
\caption{Graph of $y=z^{2m-n}(z^n-a^n)$,\,$n=5,m=4,a=1$} 
\end{figure}
Considering other solutons of the argument $2\pi j/n$, $j=0,\dots, n-1$, we get $2n$ solutions.


\noindent
(1b) Assume that $2m>n$ and $n$ is even.
The equation for the real solutions is 
\[
z^{2m-n}(z^n-a^n)=1
\]
and it has two real solutions. Thus on the lines $L_{2j\pi/n}$,  $2n/2=n$ solutions.
See Figure 5.
On the real lines $L_{(2j+1)\pi/n}$,  the equation reduces to
\[
u^{2m-n}(u^n+a^n)=1.
\]
Thus it has $2n/2=n$ solutions on these lines and altogether, we gave $2n$ solutions.
\begin{figure}[htb] 
\setlength{\unitlength}{1bp} 
\centerline{\includegraphics[width=6cm,height=4cm, angle=0]{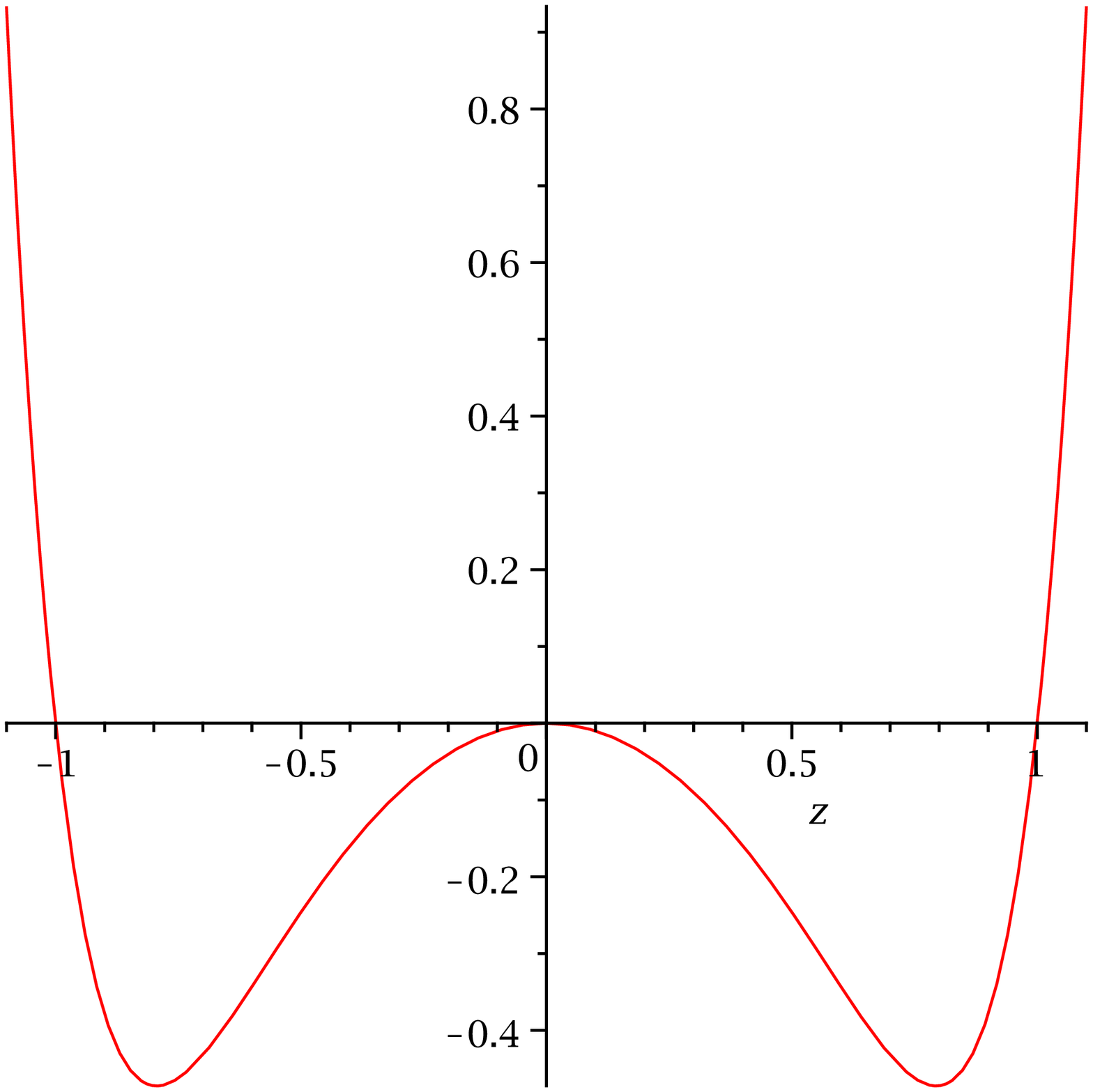}} 
\vspace{1cm}
\caption{Graph of $y=z^{2m-n}(z^n-a^n)$,\,$n=6,m=4,a=1$} 
\end{figure}
\begin{figure}[htb] 
\setlength{\unitlength}{1bp} 
\centerline{\includegraphics[width=6cm,height=4cm,angle=0]{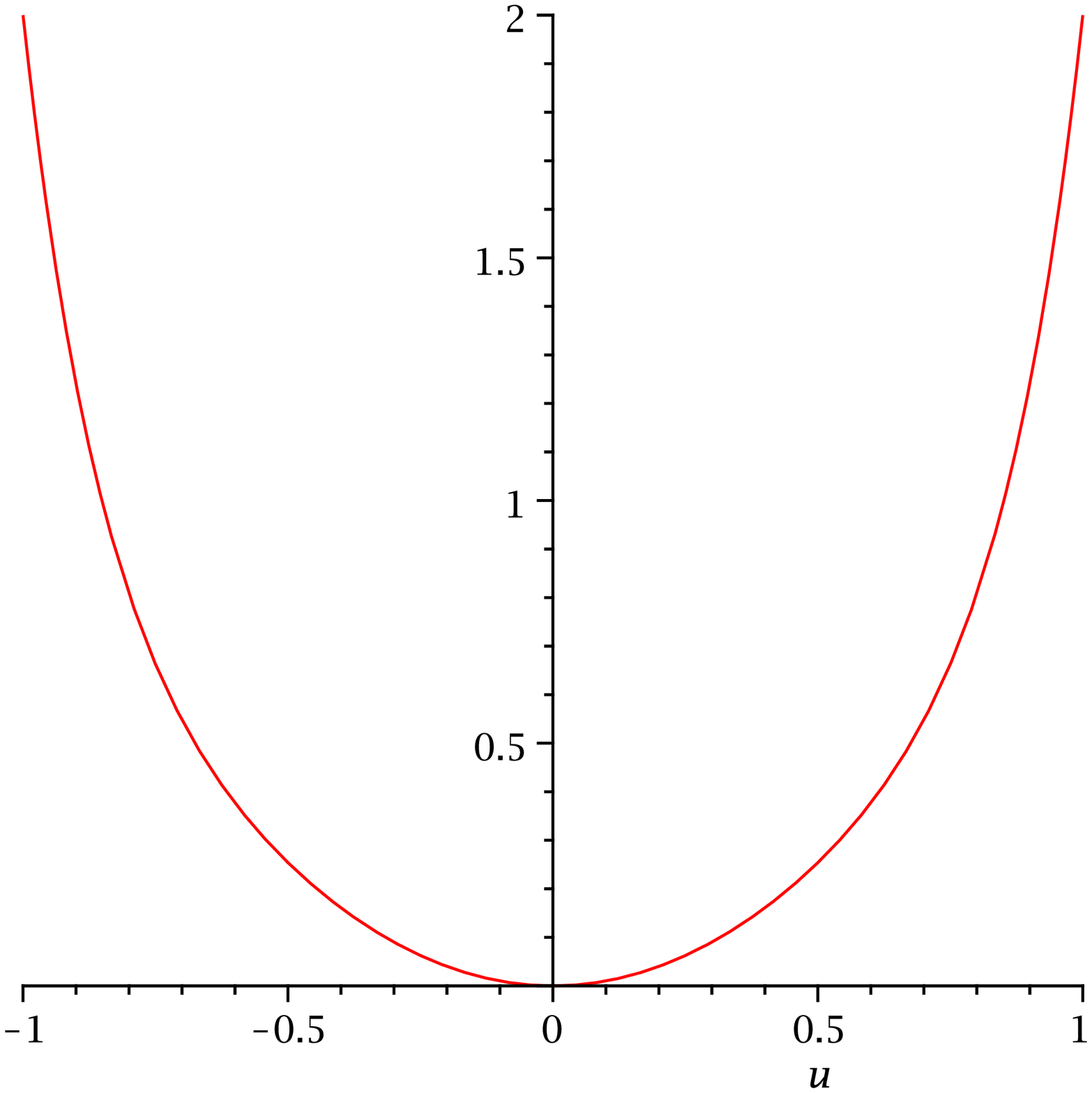}} 
\vspace{1cm}
\caption{Graph of $y=z^{2m-n}(z^n+a^n)$,\,$n=6,m=4,a=1$} 
\end{figure}
\noindent

(2) Assume that $n=2m$. Then (\ref{2m>n} ) reduces to
\[
z^{2m}-a^{2m}=1.
\]
This has two real roots on $L_0$ and thus 
 we get $2n/2=n$ roots on the lines $L(n)$.
On the lines $\arg\,z=(2j+1)\pi/n$, putting $z=u \exp(\pi/n)$,
the equation is given by  $u^n+a^n=1$. This has two roots provided $a<1$ and thus $n$ roots on $L(n)'$.
Thus altogether, we get $2n$ roots.
\end{proof}

\section{Bifurcation of the root  and the main result}
We considered  the extended  lens equation for a fixed $a>0$ as in Lemma \ref{ex-lens}. Note that $z=0$ is a root with multiplicity. 
We want to change these roots into  $2n$ regular roots using 
a small  bifurcation.
\begin{eqnarray}
&\ell_{n,m}^\eps&:=\bar z^m-\dfrac{z^{n-m}}{z^n-a^n}+\dfrac{\eps}{z^m},\,\,\eps>0.\label{bif1}
\end{eqnarray}
Note that the  mixed polynomial, given by the numerator of $\ell_{n,m}^\eps$ (by abuse of the notation, we denote this numerator also by the same notation) satisfies  
\begin{eqnarray*}
&\ell_{n,m}^\eps\in L(n_1+m;n_1,m)\subset M(n_1+m;n_1,m)\\
&\text{where}\,\, \qquad n_1:=n+m.
\end{eqnarray*}
First we observe (\ref{bif1}) implies
\[
z^n(|z|^{2m}-1+\eps)=\eps a^n+|z|^{2m}a^n
\]
which implies that
$z^{2n}$ is a positive real number as the situation before the  bifurcation.
We observe that
\begin{Proposition}
 $V(\ell_{n,m}^\eps)$ is also a subset of $\mathcal L(2n)$ and it is  $\mathbb Z/n\mathbb Z$-invariant.
\end{Proposition}
\subsection{The case $m$ is not so big} Assume that  $n>2m$ or $n_1>3m$.
The following is our main result which generalize the result of Rhie for the case $m=1$.
\begin{Theorem}\label{main1} 
\begin{enumerate}
\item Assume that $n>2m$ i.e.,  $n_1>3m$.
For a sufficiently small positive   $\eps$, $\rho(\ell_{n,m}^\eps)=5(n_1-m)$.
\item
For the case $n=2m$,  let $f_{2m}$ be as in Corollary \ref{case-n=2m}.
Then  $f_{2m} \in L(3m;2m,m)$ and $\rho(f_{2m})=5m$.
\end{enumerate}
\end{Theorem}
\begin{proof}
We prove the assertion for the case $n>2m$, the assertion for $n=2m$ is in Corollary \ref{case-n=2m}.
First observe that  $3n$ roots of $\ell_{n,m}$ are all simple.
 Put them 
$\xi_1,\dots, \xi_{3n}$. Take a small radius $r$ so that the disks $D_r(\xi_j),\,j=1,\dots, 3n$  of radius $r$ centered at $\xi_j$
are disjoint each other and  they do not contain  $0$
and the Jacobian of $\Re \ell_{n,m},\Im\,\ell_{n,m}$ has rank two everywhere on $D_r(\xi_j)$.
Then for any sufficiently small $\eps>$, there exists a single simple root in $D_r(\xi_j)$ for $\ell_{n,m}^\eps=0$.

First consider the case $n$ being odd. The real root of $\ell_{n,m}^\eps=0$ satisfies the equation
\begin{eqnarray}
&{\bar z}^m= \dfrac{z^{n-m}}{z^n-a^n}+ \dfrac{\eps}{z^m}\quad
\text{or}\\
&f_\eps:=|z|^{2m}(z^n-a^n)-z^n-\eps(z^n-a^n)=0. \label{action-inv0}
\end{eqnarray}

We consider the possible roots which bifurcate from $z=0$.
The second equation (\ref{action-inv0}) is written as 
\begin{eqnarray}\label{action-inv1}
z^{n} |z|^{2m} -z^n(1+\eps) -a^{n}(|z|^{2m}-\eps)=0
\end{eqnarray}
$(z^{2m}-\eps)=0$ has two real roots $\al_{0+}>0>\al_{0-}$.
Take a sufficiently small $s>0$  and consider the disk
$B_{0\pm}$ centered at $\al_{0\pm}$ of radius $s\eps^{1/2m}$  so that 
they do not contain  zero.

Note that $|z^{2m}-\al_{0\pm}|\ge s^{2m}\eps$ on 
$\partial B_{0\pm}$. As the other term of $f_\eps$ is of  order greater than or equal to 
$\eps^{n/2m}\ll \eps$. Thus taking $\eps$ small enough, we may assume that 
$f_\eps=0$ has a simple root inside the disk $B_{0+}$ and $B_{0-}$.

Here is another slightly  better argument.
We consider the scale change
$z=w \eps^{1/2m}$
and
put 
\[\begin{split}
\tilde f_\eps(w):=&\frac1{\eps} f(w \eps^{1/2m})\\
=& -a^n(|w|^{2m}-1)+\eps^{n/2m}w^n|w|^{2m}-\eps^{(n-2m)/2m}w^n(1+\eps).
\end{split}
\]
In this coordinate, $3n$ roots $\xi_j$ are far from the origin and we see clearly there are two roots
near $w=\pm1$ as long as $\eps$ is sufficientlt small.

We consider now roots on $L(n)$. By the $\mathbb Z/n\mathbb Z$-invariance, we have also
two  roots on each $L_{2\pi j/n}$ and thus we get $2n$ simple roots which are bifurcating from $z=0$. Thus altogether, we get $5n=5(n_1-m)$ roots.

We consider now the case $n$ being even.  Then every root of (\ref{action-inv1})  on 
$\arg\, z=2j\pi/n$ are counted twice. Thus we have $n$ roots on these real line.
In this case, there are also roots on the real lines $\arg\,z=(2j+1)\pi/n$.
In fact, put $z=u\exp\pi i/n$ in (\ref{action-inv1}).
Then the equation in $u$ takes the form:
\begin{eqnarray}\label{action-inv2}
-u^{n} |u|^{2m} +u^n(1+\eps) -a^{n}(|u|^{2m}-\eps)=0
\end{eqnarray}
This has two real roots.
Thus we found another $2n/2=n$ roots. Therefore there are $2n$ simple roots which bifurcate from $z=0$.
Thus we have $5n$ roots for $\ell_{n,m}^\eps$ in any case.
\end{proof}
\subsection{Rhie's equations}
Applying Theorem \ref{main1}, we get lens equation with maximal number of zeros
$5(n-1)$ in the form:
\[
\bar z=f_n(z),\quad f_n(z)=\frac{z^{n-2}}{z^{n-1}-a^{n-1}}-\frac{ \eps}z,\,0<\eps\ll a\ll 1
\]for $n\ge 4$.
For example, for $n=4$, we can take for example
\[
f_4(z)=\frac{z^2}{z^3-1/5}-\frac{1/800}{z}.
\]
For $n=2, 3$, the previous construction does not work and  we need a special care.
In fact, we can take $f_n$ for $n=2,3$ as follows (Compare with \cite{Bleher}):
\[
f_2(z)=\frac{z-1/30}{z^2-1/2},\quad f_3(z)=\frac{z^2-1/1000}{z^3-1/8}.
\]
\begin{figure}[htb] 
\setlength{\unitlength}{1bp} 
\vspace{.5cm}
\centerline{\includegraphics[width=8cm,height=8cm, angle=-90]{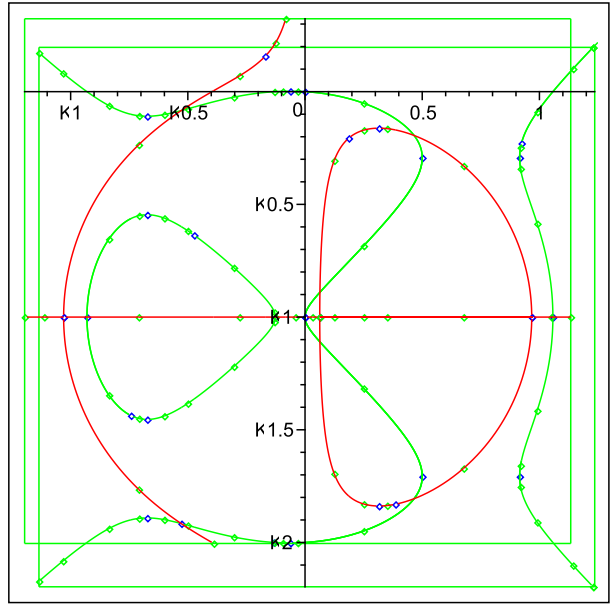}} 
\vspace{.5cm}
\caption{Graph of $f_3$} 
\end{figure}

In Figure 7,  the red curve is $\Im(numerator(\bar z-f_3(z))=0$ and the green curve is the zero set of 
$\Re(numerator(\bar z-f_3(z))=0$. The 10 intersections of green and red curves  are zeros of $\bar z-f_3(z)=0$. Graph is lifted -1 vertically.
\subsection{The case $m$ is big} Assume that $2m\ge  n>m$.
In this case, we have the following result.
\begin{Theorem}Assume that $2m\ge n>m$. Then for a sufficiently small $\eps>0$, 
$\rho(\ell_{n,m}^\eps)\ge 3n=3(n_1-m)$.
\end{Theorem}
\begin{proof}
 We have shown in Lemma \ref{ex-lens} that $L_{n,m}$ has $2n$ simple
 roots.
Thus we need to show under the bifurcation equation $\ell_{n,m}^\eps$, we
 get $n$ further roots.
\[
 {\bar z}^m- \dfrac{z^{n-m}}{z^n-a^n}- \dfrac{\eps}{z^m}=0
\]
is equivalent to
\begin{eqnarray}\label{ex-lens2}
& |z|^{2m}(z^n-a^n)-z^n-\eps(z^n-a^n)&=0\,\,\text{or}\\
&|z|^{2m}(z^n-a^n)-z^n(1+\eps)+a^n\eps&=0\label{ex-lens3}
\end{eqnarray} 
(I-1) We first consider the case
 $2m>n$ and $n$ is odd.
 $-z^n(1+\eps)+a^n\eps=0$ has one positive root $z=\be=a\root{n}\of{\eps/(1+\eps)}$.
 By a similar argument as in the previous section,
(\ref{ex-lens3}) has a simple root near $\be$.
Thus by the $\mathbb Z/n\mathbb Z$-action stability, we have $n$ simple bifurcating roots and
altogether, we have $2n+n=3n$ simple roots.

\noindent
(I-2) Assume that $2m>n$ and $n$ is even.
$-z^n(1+\eps)+a^n\eps=0$ has one positive and one negative roots.
Then by the stability (\ref{ex-lens3}) has $2n/2=n$ simple roots.
To see the roots on $L(n)'$, put $z=u\exp(i\pi/n)$.
Then (\ref{ex-lens3}) is reduced to
\[|u|^{2m}(-u^n-a^n)+u^n(1+\eps)+a^n\eps=0
\]
We see this has no real root. Thus altogteher, we have $3(n_1-m)$ roots.

\noindent
(II) Assume that $n=2m$. Then (\ref{ex-lens3}) can be written as
\begin{eqnarray}\label{n=2m2}
 |z|^{2m}z^{2m}  -z^{2m}(a^{2m}+1+\eps)+a^{2m} \eps=0
\end{eqnarray}
if $z$ real. This has two real roots and adding all roots in the lines of $L(n)$,   we get $2n/2=n$ roots.

Consider other roots on $L(n)'$.
Then we can write $z=u\zeta$ with $\zeta^n=-1$ and $u\in \mathbb R$  and 
\[
- |u|^{2m}u^{2m}+u^{2m}(1-a^{2m}+\eps)+a^{2m}\eps=0.
\]
This has no zeros near the origin,  as we have assumed $0<a<1$ to have 2n zeros in $L_{n,m}=0$.
Thus the above bifurcation equation has no real root. Thus altogether
we conclude $\rho(L_{n,m}^\eps)= 3n=3(n_1-m)$.
\end{proof}

\subsection{Application}

\subsubsection{$L^{hs}(n+m;n,m)$}
The space of harmonically splitting Lens type polynomials
apparently  can take bigger number of zeros than generalized lens polynomials.
To show this, we start from arbitrary lens equation
\[\begin{split}
\ell_n(z):=&\bar z-\frac{p(z)}{q(z)},\quad \deg_zq=n,\, \deg_zp\le n,\\
\end{split}
\] Put $k=\rho(\ell_n)$.
We assume that $0$ is not a root of $\ell_n$ for simplicity and $q(z)$ has coefficient 1 for $z^n$.
We consider its small purturbation in $L^{hs}(n+m;n,m)$:
\[
\phi_t(z):=-t \bar z^m+\ell_n(z)=-t \bar z^m+\bar z-\frac{p(z)}{q(z)},\,\,1\gg t>0.
\]
We assert
\begin{Theorem}
For sufficiently small $t>0$, $\rho(\phi_t)=k+m-1$.
\end{Theorem}
\begin{proof}
As before, we identify $\phi_t, \ell_n$ with their numerators.
For sufficiently small $t$ and for each zero root $\al$ of $\ell_n$, there exists a zero $\al'$ of $\phi_t$ in a neighborhood of  $\al$
which has the same orientation as $\al$.
For $t\ne 0$, we know that $\be(\phi_t)=n-m$ and $\be(\ell_n)=n-1$.
Here $\be(f)$ is the number of zeros of $f$  with sign. See Theorem \ref{main}.
By the assumption, $\ell_n$ has $k$ zeros, say $\al_1,\dots, \al_{k}$ and $\beta(\ell_n)=n-1$.
First we choose a positive number $R$ so that 
$1/R<|\al_j| <R$ for $j=1,\dots, 5n-5$. Thus it is clear that $\phi_t$ has $k$ zeros near each $\al_j(\eps)$ with 
the same sign as $\al_j$ in the original equation  $\ell_n=0$.
 On the other hand,
$\be(\phi_t)=n-m,\,t\ne 0$.
Thus $\phi_\eps$ has at least $m-1$ new negative zeros.

We assert that $\phi_t$ obtains exactly $m-1$ new negative  zeros near infinity.
To see this near infinity, we change the coordinate $u=1/z$ and 
consider the numerator:
$((-t/\bar u^m-1/\bar u)q(1/u)-p(1/u))\,\bar u^m u^n$. This takes the form
\[
\Phi_t=(-t+\bar u^{m-1})\, \tilde q(u)-\bar u^m\tilde p(u)
\]
where $\tilde q,\tilde p$ are polynomials defined as 
$\tilde q(u)=u^nq(1/u),\, \tilde p(u)=u^np(1/u)$.
By asumption  we can write 
\[\begin{split}
\tilde q(u)&=1+\sum_{i=1}^n b_iu^i\\
\tilde p(u)&= \sum_{i=0}^n c_iu^i.
\end{split}
\]
We will prove that  for a sufficiently small $t>0$, there exist exactly $m-1$ zeros  $u(t)$ which 
converges to 0 as $t\to 0$. 
The zero set  
\[
\{(u,t)\in \mathbb C\times \mathbb R\,|\, \Phi_t(u)=0\}
\]
 in $\mathbb C\times \mathbb R$
is a real algebraic set. Thus we  need only check the components which intersect
with $t=0$.
We use the Curve selection lemma.
Suppose that
\begin{eqnarray}
\Phi_{t(s)}(u(s))&\equiv& 0,\label{identity},\,\,
t(s)=s^a,\quad\\
 u(s)&=&\sum_{j=p}^\infty d_js^j,\,\, d_p\ne 0.
\end{eqnarray}
Note that the possible lowest order of $(-t(s)+\bar u(s)^{m-1})\tilde q(u(s))$
is $\min(a,p(m-1))$, while the lowest order of
the second term $\bar u(s)^m \tilde p(u(s))$ is  $pm$.
Thus (\ref{identity}) says 
\[a=p(m-1),\quad -1+{\overline{d}_p}^{m-1}=0.
\]
Thus we can write
\begin{eqnarray}\label{first term}
d_p=\exp(2\pi ji/(m-1)),\quad \exists j,\,0\le j\le m-2.
\end{eqnarray}
We assert that
\begin{Assertion}
 For a fixed $j$, there exist a unique $u(s)$ which satisfies (\ref{identity})
and (\ref{first term}).
\end{Assertion}
We prove the coefficients $d_j$ of $u(s)$ are uniquely determined by induction.
Put
\begin{eqnarray*}
(-t(s)+\bar u^{m-1}(s)\tilde q(u(s))&=&\sum_{\nu=p(m-1)}^\infty \ga_\nu s^\nu\\
\bar u(s)^m\tilde p(u(s))&=&\sum_{\nu=pm}^\infty \de_\nu s^\nu.
\end{eqnarray*}
We have shown $\ga_{p(m-1)}=0$ as $d_p=\exp(2\pi ji/(m-1))$.
Suppose that $d_j,\,p\le j\le \mu-1$  are uniquely determined.
We consider the coefficient of $s^{p(m-2)+\mu}$ in  (\ref{identity}).
We need to have 
\[
\ga_{p(m-2)+\mu}=\de_{p(m-2)+\mu}.
\]
Observe that 
\[
\ga_{p(m-2)+\mu}=(m-1)\overline{d}_p^{m-1}\overline{d}_\mu+r'
\]
 where $r'$ is a polynomial of coefficients
$\{\overline{d}_j,\,j\le \mu-1\}\cup\{b_j,\, j=1,\dots,n\}$. On the other hand,
$\de_{p(m-2)+\mu}$ is a polynomial of coefficients $\{\overline{d}_j,\,j\le \mu-1\}\cup\{c_j,\, j=0,\dots,n\}$.
Thus $d_\mu$ is uniquely determined by the equality
$\ga_{p(m-2)+\mu}=\de_{p(m-2)+\mu} $.
\end{proof} 
As $\Phi_t\in L^{hs}(n+m;n,m)$,  combining with Theorem \ref{Lens-range}, we obtain the following.
\begin{Corollary} \label{hs-rho} The set of the number of zeros $\rho(f)$ of  harmonically splitting lens type polynomials 
$f\in  L^{hs}(n+m;n,m)$ includes   $ \{n+m-2,n+m,\cdots, 5n+m-6\}$.
\end{Corollary}
\subsubsection{ The moduli space $M(n+m;n,m)$}
Now we consider the bigger  class of polynomials $M(n+m;n,m)\supset L^{hs}(n+m;n,m)$. 
As $\be(F)=n-m$ for $F\in M(n+m;n,m)$, the lowest possible number of zeros of a polynomial in $M(n+m;n,m)$ is $n-m$.
In fact we assert
\begin{Corollary} The set $\{\rho(f)\,|\, f\in M(n+m;n,m)\}$ includes 
$\{n-m,n-m+2,\dots, n+m-2,\dots, 5n+m-6\}$.
\end{Corollary}
\begin{proof}
By Corollary \ref{hs-rho}, it is enough to show that any of  $\{n-m,n-m+2,\dots, n+m-4\}$ can be $\rho$ of some $f\in M(n+m;n,m)$.
Let $j=n-m+2a,\, 0\le a\le m-2$.
Consider the polynomial
\[
f_a(z)=(z^{n-a}\bar z^{m-a}-1)(z^a-2)(\bar z^a-3)
\]
Then we see that $\rho(f_a)=n-m+2a$ and $f_a\in M(n+m;n,m)$.
\end{proof}
\begin{Example}Consider 
$M(5;3,2)$. The possible $\rho$ are $\{1,3,\dots, 11\}$.
For $\rho=1, 3, 5$, we can take for example mixed polynomials associated with 
thefollowing  polynomials
\[ f(z)=z^3\bar z^2-1,\, (z^2\bar z-1)(z-2)(\bar z-3),\, (z-1)(z^2-2)(\bar z^2-3).
\]
The higher values $\{7,9,11\}$ are given by 
\[
\eps \bar z^2+\bar z-\frac{p(z)}{q(z)},\,\deg\,p(z)\le 3,\,\deg\, q(z)=3,\,\eps\ll 1
\]
where $\bar z-\frac{p(z)}{q(z)}=0$ is a lens type equation with $\rho=6,8,10$.
\end{Example}
\subsection{Further remark.}
\subsubsection{$L(n+2m;n+m,m)$ with $2m<n$} We observe that in Theorem \ref{main1}, 
$\rho(\ell_{n,m}^\eps)=5(n_1-m)$ with $n_1=n+m$
which is exactly the optimal upper bound for $m=1$. Thus we may expect that the number $5(n_1-m)$
might be optimal upper bound for the polynomials in $L(n_1+m;n_1,m)$. However in the proof for $m=1$,  a
result
about an attracting or rationally neutral fixed points in  complex dynamics played an important role
 and the the argument there does not apply directly in our case.

\subsubsection{$L(2m+m;2m,m)$} Our poynomial $\ell_{2m,m}^\eps$ is not good enough. We have seen  in Corollary \ref{case-n=2m} that 
the mixed polynomial $f_{2m}$ has $5m$ zeros, while our polynomial $\ell_{m,m}^\eps$ has only 
$3m$ zeros.
\begin{Problem}\begin{itemize}
\item Determine the upper bound of $\rho$ for $L(n+m;n,m)$ for $n>3m$.
\item Determine the posssible number of $\rho$ for $L(n+m;n,m)$.
Is it $\{n-m+2k\,|\, 0\le k\le 2n-2m\}$?
\item
Determine the upper bound  of $\rho$ for 
 $L^{hs}(n+m;n,m)$ or 
$M(n+m;n,m)$.
\item Are the subspaces of the moduli $L(n+m;n,m),\,L^{hs}(n+m;n,m),\, M(n+m;n,m)$ with a fixed $\rho$ connected?
If not, give an example.
\end{itemize}
\end{Problem}
\def\cprime{$'$} \def\cprime{$'$} \def\cprime{$'$} \def\cprime{$'$}
  \def\cprime{$'$} \def\cprime{$'$} \def\cprime{$'$} \def\cprime{$'$}

\end{document}